\documentclass[amstags,fleqn]{article}

\usepackage{url}

\usepackage{amssymb}

\usepackage{xspace}

\usepackage{amsmath}
\usepackage{amsthm}

\theoremstyle{plain}
\newtheorem{theorem}{Theorem}[section]
\newtheorem{lemma}[theorem]{Lemma}

\theoremstyle{definition}
\newtheorem{definition}[theorem]{Definition}

\theoremstyle{definition}
\newtheorem{example}[theorem]{Example}
\newtheorem{remark}[theorem]{Remark}

\newcommand{\opa}{\diamond} 
\newcommand{\opb}{\otimes}

\newcommand{\Rdisjoint}{\textnormal{(R1)}\xspace}
\newcommand{\RtripA}{\textnormal{(R2)}\xspace}
\newcommand{\RtripB}{\textnormal{(R3)}\xspace}

\usepackage{ifpdf}

\ifpdf
        \usepackage[pdftex]{graphicx}
\else
    \usepackage{epsfig}
        \usepackage{graphicx}
\fi

\title{Latin trades and simplicial complexes}

\author{
Carlo H\"{a}m\"{a}l\"{a}inen\footnote{Supported by Eduard \v Cech center, grant LC505.}  \\
Department of Mathematics \\ Charles University \\ Sokolovsk\'a 83 \\ 186 75 Praha 8 \\ Czech Republic \\
{\texttt carlo.hamalainen@gmail.com}}

\begin{document}

\maketitle

\begin{abstract}
In this note we introduce the concept of the {\em trade space} of a latin square.
Computations using Sage and the GAP package 
Simplicial Homology
are presented.
\end{abstract}

\section{Introduction}

We first introduce the concept of a latin square and a latin bitrade.
For a more detailed exposition and literature survey,
see~\cite{hamalainen2007}.

\begin{definition}\label{defnLatinSquare}
Let $N$ be a fixed set of size $n > 0$.
A {\em latin square\/} 
$L$ of {\em order $n$\/} is an $n \times n$ array with
rows and columns indexed by $N$,
and entries from the set $N$.
Further, each $e \in N$ appears exactly once in each row
and exactly once in
each column.
A {\em partial latin square\/} 
of order $n$
is an $n \times n$ array where each $e \in N$ occurs at most once
in each row and at most once in each column. 
\end{definition}
Usually for index set $N$ we will use
$[n] = \{ 1,\, 2, \dots,\, n \}$ 
or sometimes
$[n]_0 = \{ 0,\, 1, \dots,\, n-1 \}$ when working with 
modulo arithmetic.  
Note that a latin square is a partial latin square with no empty cells.
A latin square $L$ may also be
represented as a set of ordered triples, where
$(r,c,e) \in L$ denotes the fact that symbol $e$ appears in the cell at
row $r$, column $c$, of $L$.  Alternatively we write
$L^{\opa}$ with binary operation $\opa$ such that
$r \opa c = e$ if and only if $(r,c,e) \in L$. Similarly, a partial
latin square $P$ may be written as $P^{\opa}$.
We use setwise and binary operator notation interchangeably.

The {\em size\/} of a partial latin square $P$ is the number of filled cells, denoted
by $\left| P \right| = \left| \{ (r,c,e) \mid (r,c,e) \in P \} \right| $.
A partial latin square $P'$ is {\em contained in}, or is a 
{\em (partial) subsquare of\/} $P$ if and only if $P' \subseteq P$.

\subsection{Latin bitrades and trades}

\begin{definition}\label{defnBitradeLatin}
A {\em latin bitrade\/} $(T^{\opa},\, T^{\opb})$ is a pair of partial
latin squares such that:
\begin{enumerate}

\item $\{ (i,\,j) \mid (i,\,j,\,k) \in T^{\opa} \mbox{ for some symbol $k$} \} 
\newline
= \{ (i,\,j) \mid (i,\,j,\,k') \in T^{\opb} \mbox{ for some symbol $k'$} \}$;

\item for each $(i,\,j,\,k) \in T^{\opa}$ and $(i,\,j,\,k') \in T^{\opb}$, $k
\neq k'$;

\item the symbols appearing in row $i$ of
$T^{\opa}$
are the same as those of row $i$ of
$T^{\opb}$; 
the symbols appearing in column $j$ of
$T^{\opa}$
are the same as those of column $j$ of
$T^{\opb}$.

\end{enumerate}
\end{definition}

The following definition is equivalent to
Definition~\ref{defnBitradeLatin}:

\begin{definition}\label{defnBitradeA123}
A {\em latin bitrade\/} $(T^{\opa},\, T^{\opb})$ is a pair of partial
latin squares $T^{\opa}$, $T^{\opb} \subseteq A_1 \times A_2 \times A_3$
such that:
\begin{itemize}
\item[\Rdisjoint] $T^{\opa} \cap T^{\opb} = \emptyset$;

\item[\RtripA] for all $(a_1,\, a_2,\, a_3) \in T^{\opa}$ and all $r$,
$s \in \{1,\, 2,\, 3\}$,
$r \neq s$, there exists a unique $(b_1,\, b_2,\, b_3) \in T^{\opb}$
such that $a_r=b_r$ and $a_s=b_s$;

\item[\RtripB] for all $(a_1,\, a_2,\, a_3) \in T^{\opb}$ and all $r$,
$s \in \{1,\, 2,\, 3\}$,
$r \neq s$, there exists a unique $(b_1,\, b_2,\, b_3) \in T^{\opa}$
such that $a_r=b_r$ and $a_s=b_s$.

\end{itemize}
\end{definition}

Note that \RtripA and \RtripB imply that each row (column) of
$T^{\opa}$ contains the same subset of $A_3$ as the corresponding
row (column) of $T^{\opb}$.  Since all of the bitrades in this
dissertation are latin bitrades, we usually shorten `latin bitrade'
to just `bitrade.'

For a bitrade $(T^{\opa},\, T^{\opb})$ we refer to $T^{\opa}$ as the {\em
trade}, and $T^{\opb}$ as the {\em disjoint mate}. A particular trade
may have more than one disjoint mate.  

\begin{example}
Consider the following latin squares:
\begin{equation}
L^{\opa} = 
\begin{array}{c|cccccc}
\opa & 0 & 1 & 2 & 3 & 4 & 5 \\
\hline
0 &  0 &1 &3 &2 &4 &5 \\
1 &  5 &2 &4 &3 &1 &0 \\
2 &  2 &3 &5 &4 &0 &1 \\
3 &  3 &4 &1 &0 &5 &2 \\
4 &  4 &5 &0 &1 &2 &3 \\
5 &  1 &0 &2 &5 &3 &4
\end{array}
\qquad
L^{\opb} = 
\begin{array}{c|cccccc}
\opb & 0 & 1 & 2 & 3 & 4 & 5 \\
\hline
0 & 0 &3 &1 &2 &4 &5 \\
1 & 5 &2 &3 &4 &1 &0 \\
2 & 2 &4 &5 &3 &0 &1 \\
3 & 3 &1 &4 &0 &5 &2 \\
4 & 4 &5 &0 &1 &2 &3 \\
5 & 1 &0 &2 &5 &3 &4
\end{array}
\end{equation}
One possible bitrade $(T^{\opa},\, T^{\opb})$
where 
$T^{\opa} \subseteq L^{\opa}$
and
$T^{\opb} \subseteq L^{\opb}$
is shown below:
\begin{equation}
T^{\opa} = \begin{array}{c|ccccccc}
\opa  & 0 & 1 & 2 & 3 & 4 & 5 \\
\hline 
0 & \phantom{5} &1 &3 &\phantom{5} &\phantom{5} &\phantom{5} \\
1 & \phantom{5} &\phantom{5} &4 &3 &\phantom{5} &\phantom{5} \\
2 & \phantom{5} &3 &\phantom{5} &4 &\phantom{5} &\phantom{5} \\
3 & \phantom{5} &4 &1 &\phantom{5} &\phantom{5} &\phantom{5} \\
4 & \phantom{5} &\phantom{5} &\phantom{5} &\phantom{5} &\phantom{5} &\phantom{5} \\
5 & \phantom{5} &\phantom{5} &\phantom{5} &\phantom{5} &\phantom{5} &\phantom{5} 
\end{array}
\qquad
T^{\opb} = \begin{array}{c|ccccccc}
\opb  & 0 & 1 & 2 & 3 & 4 & 5 \\
\hline 
0 &\phantom{5} &3 &1 &\phantom{5} &\phantom{5} &\phantom{5} \\
1 &\phantom{5} &\phantom{5} &3 &4 &\phantom{5} &\phantom{5} \\
2 &\phantom{5} &4 &\phantom{5} &3 &\phantom{5} &\phantom{5} \\
3 &\phantom{5} &1 &4 &\phantom{5} &\phantom{5} &\phantom{5} \\
4 &\phantom{5} &\phantom{5} &\phantom{5} &\phantom{5} &\phantom{5} &\phantom{5} \\
5 &\phantom{5} &\phantom{5} &\phantom{5} &\phantom{5} &\phantom{5} &\phantom{5}
\end{array}
\end{equation}
So
$L^{\opb} = \left( L^{\opa} \setminus T^{\opa} \right) \cup T^{\opb}$
and
$L^{\opa} = \left( L^{\opb} \setminus T^{\opb} \right) \cup T^{\opa}$.
\end{example}

\begin{example}\label{exBitradeZ3}
The following bitrade is simply a cyclic row-shift of $T^{\opa} =
\mathbb{Z}_3$,
the integers under addition modulo~$3$:
\begin{equation*}
T^{\opa} = \begin{array}{c|ccc}
\opa  & 0 & 1 & 2 \\
\hline 
0 & 0 & 1 & 2 \\
1 & 1 & 2 & 0 \\
2 & 2 & 0 & 1
\end{array}
\qquad
T^{\opb} = \begin{array}{c|ccc}
\opb  & 0 & 1 & 2 \\
\hline 
0 & 1 & 2 & 0 \\
1 & 2 & 0 & 1 \\
2 & 0 & 1 & 2
\end{array}
\end{equation*}
 
\end{example}

\begin{example}\label{exSmallExampleTrade} 
Here is a larger bitrade:
\begin{equation}
	\begin{array}{c|cccc}
	\opa & 1 & 2 & 3 & 4 \\
	\hline 1 & 1 & 4 & 3 & 2 \\
	2 & 4 & 3 &   &   \\
	3 &   &   & 2 & 1 \\
	4 & 3 & & 1 \\
	\end{array}
\qquad
	\begin{array}{c|cccc}
	\opb & 1 & 2 & 3 & 4 \\
	\hline 1 & 4 & 3 & 2 & 1 \\
	2 & 3 & 4 &   &   \\
	3 &   &   & 1 & 2 \\
	4 & 1 & & 3 \\
	\end{array}
\label{eqn:smallExampleTrade}
\end{equation}
\end{example}

\subsection{Latin critical sets}

\begin{definition}\label{defnCriticalSet}
A partial latin square $C \subseteq L$ is a {\em critical set\/} if
\begin{enumerate}

\item $C$ has unique completion to $L$; and\label{defnCS1}

\item no proper subset of $C$ satisfies~\ref{defnCS1}.
\end{enumerate}
\end{definition}

\begin{example}
Latin square $\mathbb{Z}_2$ and a critical set:
\[
	\begin{array}{|c|c|}
	\hline 0&1\\
	\hline 1&0\\
	\hline 
	\end{array}
\qquad
	\begin{array}{|c|c|}
	\hline ~&~\\
	\hline ~&0\\
	\hline 
	\end{array}
\]
\end{example}

\begin{example}
Latin square $L_3$ and critical set $P_3 \subset L_3$:
\[
	\begin{array}{|c|c|c|c||c|c|c|c|}
	\hline 0&1&2&3&4&5&6&7\\
	\hline 1&0&3&2&5&4&7&6\\
	\hline 2&3&0&1&6&7&4&5\\
	\hline 3&2&1&0&7&6&5&4\\
	\hline
	\hline 4&5&6&7&0&1&2&3\\
	\hline 5&4&7&6&1&0&3&2\\
	\hline 6&7&4&5&2&3&0&1\\
	\hline 7&6&5&4&3&2&1&0\\
	\hline 
	\end{array}
\qquad
	\begin{array}{|c|c|c|c||c|c|c|c|}
	\hline 0&1&2&3&4&5&6&~\\
	\hline 1&0&3&2&5&4&~&~\\
	\hline 2&3&0&1&6&~&4&~\\
	\hline 3&2&1&0&~&~&~&~\\
	\hline
	\hline 4&5&6&~&0&1&2&~\\
	\hline 5&4&~&~&1&0&~&~\\
	\hline 6&~&4&~&2&~&0&~\\
	\hline ~&~&~&~&~&~&~&~\\
	\hline 
	\end{array}
\]
\end{example}

Fix a latin square $L$ and define the {\em trade space}
\[
\mathcal{T}_L = \{ T^{\opa} \mid 
T^{\opa} \subset L \textnormal{ and } 
(T^{\opa},\, T^{\opb}) \textnormal{ is a bitrade} \}.
\]

\begin{lemma}
Let $C$ be a critical set of the latin square $L$. Then for any 
$T \in \mathcal{T}_L$, there exists $c \in C$ such that $c \in T$.  
\end{lemma}

An long-standing open conjecture is that the size of the smallest
critical set in any latin square of order $n$ is $\lfloor n^2/4
\rfloor$. See \cite{MR2330083}, \cite{MR2369994} and \cite{MR2136056}
for recent work on this problem.

\subsection{Simplicial complexes}

We follow the presentation of Faridi~\cite{faridi2002}.
A {\em facet} of a simplicial complex is a maximal face under
inclusion.  A {\em vertex cover} $A$ of a simplicial complex $\Delta$
is a subset of vertices of $\Delta$ such that any facet is adjacent to some $a
\in A$. A {\em minimal vertex cover} $A$ is a vertex cover of $\Delta$
such that no proper subset of $A$ is a vertex cover.

Let $L$ be a latin square of order $n$. Form the simplicial complex
$\Delta(L)$ with vertices corresponding to trades $T \subset L$, and
edges $(T,T')$ for trades $T$, $T' \subset L$ such that $T \cap T'
\neq \emptyset$, and so on for faces of higher dimensions.

\begin{lemma}
For each critical set $C$ of $L$ there exists a minimal vertex cover
$A_C$ of $\Delta(L)$ and $\left| C \right| = \left| A_C \right|$.
\end{lemma}

\begin{proof}
Let $C = \{c_1, \dots, c_m\}$ be a critical set in $L$. Create a family
of $m$ sets
\[
S_i = \{ T \subseteq L \mid c_i \in T \}
\]
where $T$ is a trade in $L$. We will show that the $S_i$ have a system
of distinct representatives. Choose any $\ell$ of the sets,
$S_{i_1}$, \dots, $S_{i_{\ell}}$ and suppose that their union
$\mathcal{S}$ contains 
$l < \ell$ trades. By the pigeonhole principle there is an entry
$c_{i_u}$ such that each trade that intersects 
$c_{i_u}$ also intersects some
$c_{i_v}$, $v \neq u$. So the set $C \setminus \{ c_{i_u} \}$ is a
critical set, contradicting the minimality of $C$. Call the system of
distinct representatives $A_C$.

To see that $A_C$ is a vertex cover of $\Delta(L)$, let 
$F$ be some facet of $\Delta(L)$. Pick any $T \in F$. Then there exists
some $c_i \in C$ such that $c_i \in T$. We chose some $\overline{T}$ as
the representative of $S_i$. There is an edge $(T, \overline{T})$ due to
the common element $c_i$, so $F$ is adjacent to $\overline{T} \in A_C$.

For minimality, suppose that 
$A_C \setminus \{ T \}$ is a vertex cover for some trade $T$ in $L$.
Suppose that $T$ was chosen as the representative of $S_i$. Let
$U$ be any trade in $L$ and consider $C' = C \setminus \{ c_i \}$.
If $c_i \notin U$ then some other entry of $C$ covers $U$ since $C$ is a
critical set. Otherwise, suppose that 
$c_i \in U$ then consider the facet $F$ containing $T$ and $U$. By
assumption there is some $T' \in A_C \setminus \{T\}$ such that
$T' \in F$. Then $T'$ must be the representative for $S_j$
(where $j \neq i$ since $A_C$ is a system of distinct representatives)
so $c_j \in U$. In this way any trade $U$ is covered
by some element of $C' = C \setminus \{ c_i \}$, contradicting the
minimality of $C$.  
\end{proof}

\begin{lemma}
For each minimal vertex cover $A$ of $\Delta(L)$ there exists a critical
set $C_A$ of $L$
such that $\left| C_A \right| = \left| A \right|$.
\end{lemma}

\begin{proof}
Let $A$ be a minimal vertex cover of $\Delta(L)$.  For each $T \in
A$, let $x_T$ be an entry of $T$ contained in the intersection of
the trades that make up the facet containing $T$.  We show that $X$
is a critical set. Suppose that $U$ is some trade in $L$ and no
$x_T$ is in $U$. The vertex $U$ is adjacent to some facet $F$ and
by definition some $T' \in A$ is adjacent to this $F$. Then by
definition $x_{T'} \in U$, a contradiction. Thus
$X$ is a uniquely completable partial latin square.

For minimality, suppose that the set $X' = X \setminus \{x_T\}$ is
uniquely completable for some
$x_T \in X$. By the previous lemma, $X'$ is equivalent to a vertex
cover $A'$ of the same size, contradicting the minimality of
$A$. So $X$ is minimal and $\left| X \right| = \left| A \right|$.
\end{proof}

\begin{remark}
A critical set of $L$ is equivalent to a vertex cover of $\Delta(L)$.
The simplicial complex $\Delta(L)$ captures information about the
intersections of all trades. The critical sets correspond to minimal
vertex covers.

Faridi's paper (\cite{faridi2002}, Proposition~1) shows that a minimum vertex cover is
the same as a minimal prime ideal in the {\em facet ideal}
$\mathcal{F}(\Delta)$ in the polynomial ring
$R = k[x_1, \dots, x_n]$. In future research we would like to calculate 
more detailed information about the facet ideal and its minimal prime ideals.
\end{remark}

\section{Computations}

In this section we present computations of the {\em reduced homology
groups} $H_k$ of the facet ideal $\Delta(L)$ for various latin squares
$L$ of small order. Informally, the group $H_k$ indicates the number of
$k$-dimension `holes' in the simplicial complex at hand.
See~\cite{MR1867354} for more information about homology theory.
We use a Sage script~\cite{code} and the GAP package Simplicial Homology to
calculate the reduced homology groups.
For more information about Sage see~\cite{sage}. The Simplicial Homology
package is available at \url{http://www.cis.udel.edu/~dumas/Homology/}.

In this section, $B_n$ denotes the addition table for integers modulo
$n$ (also known as the back-circulant square). Also, let
$L_s$ denote the latin square corresponding to the group table of the
elementary abelian 2--group of order $2^s$. Formally, these latin
squares are defined by
\[
L_1 = \begin{tabular}{|c|c|c|c|c|c|}
\hline 0 & 1 \\
\hline 1 & 0 \\
\hline 
\end{tabular} \\
\]
and for $s \geq 2$,
\begin{alignat*}{3}
L_s &= L_1 \times L_{s-1} &= &\{ (x,y;z), (x,y+n/2;z+n/2),
(x+n/2,y;z+n/2),\\
&&	& (x+n/2,y+n/2;z) \mid (x,y;z) \in L_{s-1} \}
\end{alignat*}
For example,
\[
L_3 =
	\begin{tabular}{|c|c|c|c||c|c|c|c|}
	\hline 0&1&2&3&4&5&6&7\\
	\hline 1&0&3&2&5&4&7&6\\
	\hline 2&3&0&1&6&7&4&5\\
	\hline 3&2&1&0&7&6&5&4\\
	\hline
	\hline 4&5&6&7&0&1&2&3\\
	\hline 5&4&7&6&1&0&3&2\\
	\hline 6&7&4&5&2&3&0&1\\
	\hline 7&6&5&4&3&2&1&0\\
	\hline 
	\end{tabular}
\]

We now present computations for small orders:

\begin{center}
\begin{tabular}{|c|c|}
\hline Latin square & Homology groups \\
\hline $B_3$ & $[0, 0, 0, 10, 0, 0]$ \\
\hline $B_4$ & $[0, 0, 0, 0, 0, 0, 0, 0, 0, 0, 0, 0, 0, 0, 1]$ \\
\hline $L_2$ & $[0, 0, 0, 0, 0, 0, 0, 0, 0, 0, 0, 8, 9, 0 ]$ \\
\hline
\end{tabular}
\end{center}

So for $B_3$ we have
$\tilde H_0 = 0$,
$\tilde H_1 = 0$,
$\tilde H_2 = 0$,
$\tilde H_3 = 10$,
$\tilde H_4 = 0$,
$\tilde H_5 = 0$.

\subsection{Intercalate homology}

An {\em intercalate} is a latin trade of size four. Intercalates are
interesting because they are the simplest (smallest) type of latin trade
and in fact any trade can be written as a sum of
intercalates~\cite{MR1880972}. Another interesting question is whether
there are latin squares with as many intercalates as possible, or as few
as is possible (see for example~\cite{MR1700841}).

For a latin square $L$ we define the 
define the {\em intercalate trade space} as
\[
\mathcal{I}_L = \{ T^{\opa} \mid 
T^{\opa} \subset L,\,
\left| T^{\opa} \right| = 4 \textnormal{ and } 
(T^{\opa},\, T^{\opb}) \textnormal{ is a bitrade} \}.
\]
We can then create a simplicial complex where vertices are intercalates,
and higher dimension simplices are collections intercalates that
intersect in some common point. Here is the homology group information
for the elementary abelian 2--group of various orders:

\begin{center}
\begin{tabular}{|c|c|}
\hline Latin square & Intercalate homology groups \\
\hline $L_1$ & $[0, 0, 0, 0]$ \\
\hline $L_2$ & $[0, 21, 0, 0]$ \\
\hline $L_3$ & $[0, 273, 0, 0]$ \\
\hline $L_4$ & $[0, 2625, 0, 0]$ \\
\hline $L_5$ & $[0, 22785, 0, 0]$ \\
\hline $L_6$ & $[0, 189441, 0, 0]$ \\
\hline
\end{tabular}
\end{center}

For the back-circulant $B_n$ with $n \geq 4$ and $n$ even we have
$\tilde H_1 = 0$,
$\tilde H_2 = 0$,
$\tilde H_3 = 0$,
and only $\tilde H_0 \neq 0$. The values for 
$\tilde H_0$ are given below:

\begin{quote}
3, 8, 15, 24, 35, 48, 63, 80, 99, 120, 143, 168, 195, 224, 255, 288,
323, 360, 399, 440, 483, 528, 575, 624, 675, 728, 783, 840, 899, 960,
1023, 1088, 1155, 1224, 1295, 1368, 1443, 1520.  1599, 1680, 1763,
1848, 1935, 2024, 2115, 2208, 2303, 2400, 2499, 2600, 2703, 2808,
2915, 3024, 3135, 3248, 3363, 3480, 3599, 3720, 3843, 3968, 4095,
4224, 4355, 4488, 4623, 4760, 4899, 5040, 5183, 5328, 5475, 5624,
5775, 5928, 6083, 6240, 6399, 6560, 6723, 6888, 7055, 7224, 7395, 7568,
7743, 7920, 8099, 8280, 8463, 8648, 8835, 9024, 9215, 9408, 9603, 9800.
\end{quote}


\begin{thebibliography}{10}

\bibitem{code}
Tradespace source code.
\newblock \url{http://bitbucket.org/carlohamalainen/tradespace}.

\bibitem{MR2330083}
Nicholas~J. Cavenagh.
\newblock A superlinear lower bound for the size of a critical set in a {L}atin
  square.
\newblock {\em J. Combin. Des.}, 15(4):269--282, 2007.

\bibitem{MR1700841}
P.~Danziger and E.~Mendelsohn.
\newblock Intercalates everywhere.
\newblock In {\em Geometry, combinatorial designs and related structures
  ({S}petses, 1996)}, volume 245 of {\em London Math. Soc. Lecture Note Ser.},
  pages 69--88. Cambridge Univ. Press, Cambridge, 1997.

\bibitem{MR2369994}
Diane Donovan, James LeFevre, and G.~H.~John van Rees.
\newblock On the spectrum of critical sets in {L}atin squares of order {$2\sp
  n$}.
\newblock {\em J. Combin. Des.}, 16(1):25--43, 2008.

\bibitem{MR1880972}
Diane Donovan and E.~S. Mahmoodian.
\newblock An algorithm for writing any {L}atin interchange as a sum of
  intercalates.
\newblock {\em Bull. Inst. Combin. Appl.}, 34:90--98, 2002.

\bibitem{faridi2002}
Sara Faridi.
\newblock The facet ideal of a simplicial complex, 2002.

\bibitem{MR2136056}
M.~Ghandehari, H.~Hatami, and E.~S. Mahmoodian.
\newblock On the size of the minimum critical set of a {L}atin square.
\newblock {\em Discrete Math.}, 293(1-3):121--127, 2005.

\bibitem{hamalainen2007}
Carlo H\"{a}m\"{a}l\"{a}inen.
\newblock {\em Latin Bitrades and Related Structures}.
\newblock {PhD} in {M}athematics, Department of Mathematics, The University of
  Queensland, 2007.
\newblock http://carlo-hamalainen.net/phd/hamalainen-20071025.pdf.

\bibitem{MR1867354}
Allen Hatcher.
\newblock {\em Algebraic topology}.
\newblock Cambridge University Press, Cambridge, 2002.

\bibitem{sage}
W.\thinspace{}A. Stein et~al.
\newblock {\em {S}age {M}athematics {S}oftware ({V}ersion 4.0.1)}.
\newblock The Sage~Group, 2009.
\newblock {\tt http://www.sagemath.org}.

\end{thebibliography}

\end{document}